\documentclass[reqno]{article}
\usepackage{amsmath,amsthm,amssymb}
\usepackage{xcolor}
\usepackage[left=1.35in, right=1.35in, top=.9in, bottom=1.1in]{geometry}
\newtheorem{theorem}{Theorem}[section]

\newtheorem{proposition}[theorem]{Proposition}
\newtheorem{corollary}[theorem]{Corollary}
\newtheorem{example}[theorem]{Example}

\newtheorem{question}[theorem]{Question}
\newtheorem{observation}[theorem]{Observation}

\newtheorem{problem}[theorem]{Problem}
\newtheorem{claim}{Claim}
\newtheorem{fact}[theorem]{Fact}
\newcommand{\thistheoremname}{}
\newtheorem*{genericthm*}{\thistheoremname}
\newenvironment{namedthm*}[1]
  {\renewcommand{\thistheoremname}{#1}%
   \begin{genericthm*}}
  {\end{genericthm*}}

\def\beq{\begin{equation}}\def\eeq{\end{equation}}
\def\beqn{\begin{eqnarray}}\def\eeqn{\end{eqnarray}}

\newcommand{\floor}[1]{\left\lfloor#1\right\rfloor}
\newcommand{\ceiling}[1]{\left\lceil#1\right\rceil}

\newcommand{\RC}{R^{1C}}

\newcommand{\RM}{R^{M}}
\newcommand{\PM}{R^{PM}}

\newcommand{\lf}{path-matching}
\newcommand{\lfs}{path-matchings}

\begin{document}
\title{Ramsey numbers of path-matchings, covering designs, and 1-cores}
\author{Louis DeBiasio\thanks{Department of Mathematics, Miami University, Oxford, Ohio. \texttt{debiasld@miamioh.edu}} \thanks{Research supported in part by Simons Foundation Collaboration Grant \#283194}\and Andr\'as Gy\'arf\'as\thanks{Alfr\'ed R\'enyi Institute of Mathematics, Hungarian Academy of Sciences, Budapest, P.O. Box 127, Budapest, Hungary, H-1364. \texttt{gyarfas.andras@renyi.mta.hu}, \texttt{sarkozy.gabor@renyi.mta.hu}} \thanks{Research supported in part by
NKFIH Grant No. K116769.} \and G\'{a}bor N. S\'ark\"ozy\footnotemark[3]
\thanks{Computer Science Department, Worcester Polytechnic Institute, Worcester, MA.} \thanks{Research supported in part by
NKFIH Grants No. K116769, K117879.}
}

\maketitle

\begin{abstract}
A path-matching of order $p$ is a vertex disjoint union of nontrivial paths spanning $p$ vertices. Burr and Roberts, and Faudree and Schelp determined the 2-color Ramsey number of path-matchings. In this paper we study the multicolor Ramsey number of path-matchings. Given positive integers $r, p_1, \dots, p_r$, define $R^{PM}(p_1, \dots, p_r)$ to be the smallest integer $n$ such that in any $r$-coloring of the edges of $K_n$ there exists a path-matching of color $i$ and order at least $p_i$ for some $i\in [r]$. Our main result is that for $r\geq 2$ and $p_1\geq \dots\geq p_r\geq 2$, if $p_1\geq 2r-2$, then \[R^{PM}(p_1, \dots, p_r)= p_1- (r-1) + \sum_{i=2}^{r}\left\lceil\frac{p_i}{3}\right\rceil.\] Perhaps surprisingly, we show that when $p_1<2r-2$, it is possible that $R^{PM}(p_1, \dots, p_r)$ is larger than $p_1- (r-1) + \sum_{i=2}^{r}\left\lceil\frac{p_i}{3}\right\rceil$, but in any case we determine the correct value to within a constant (depending on $r$); i.e. \[p_1- (r-1) + \sum_{i=2}^{r}\left\lceil\frac{p_i}{3}\right\rceil \leq R^{PM}(p_1, \dots, p_r)\leq \left\lceil p_1-\frac{r}{3}+\sum_{i=2}^r\frac{p_i}{3}\right\rceil.\] As a corollary we get that in every $r$-coloring of $K_n$ there is a monochromatic path-matching of order at least $3\left\lfloor\frac{n}{r+2}\right\rfloor$, which is essentially best possible. We also determine $R^{PM}(p_1, \dots, p_r)$ in all cases when the number of colors is at most 4.

The proof of the main result uses a minimax theorem for path-matchings derived from a result of Las Vergnas (extending Tutte's 1-factor theorem) to show that the value of $R^{PM}(p_1, \dots, p_r)$ depends on the block sizes in covering designs (which can be also formulated in terms of monochromatic $1$-cores in colored complete graphs). While block sizes in covering designs have been studied intensively before, they seem to have only been studied in the uniform case (when all block sizes are equal). Then we obtain the result above by giving estimates on the block sizes in covering designs in the arbitrary (non-uniform) case.
\end{abstract}

\section{Introduction}

One of the seminal results in graph-Ramsey theory is the following theorem of Cockayne and Lorimer \cite{CL} which gives the $r$-color Ramsey number of a matching.  Given positive integers $r, p_1, \dots, p_r$ let $\RM(p_1, \dots, p_r)$ be the smallest integer $n$ such that in every $r$-coloring of the edges of $K_n$, there exists a matching of color $i$ and order at least $p_i$ for some $i\in [r]$.

\begin{theorem}[Cockayne, Lorimer \cite{CL}]\label{colo} Let $r\geq 2$ and let $p_1\ge p_2\ge \dots \ge p_r\geq 2$. Then $$\RM(p_1, \dots, p_r)= p_1-(r-1) + \sum_{i=2}^{r} \ceiling{\frac{p_i}{2}}.$$
\end{theorem}

Theorem \ref{colo} is sharp, shown by the coloring
$[p_1-1, \ceiling{\frac{p_2}{2}}-1,\dots, \ceiling{\frac{p_r}{2}}-1]$ which is defined as follows: Given integers $r\geq 2$ and $t_1, \dots, t_r\geq 0$, let $n=t_1+\dots+t_r$ and define $[t_1,t_2,\dots, t_r]$ to be the $r$-coloring of $K_n$ obtained by partitioning $V(K_n)$ as $\{A_1, \dots, A_r\}$ such that
$|A_i|=t_i$ for all $i\in [r]$, and coloring every edge $\{x,y\}$ with the maximum $j$ for which $\{x,y\}$ has a non-empty intersection with $A_j$.

We denote by $P_k$ the path with $k$ vertices and define a \emph{\lf{}} as a vertex disjoint union of paths, each with at least 2 vertices. The \emph{order} of a \lf{} $P$ is $|V(P)|$; i.e.\ the number of vertices spanned by $P$.  A \lf{} can clearly be written as vertex disjoint union of $P_2$-s and $P_3$-s. Thus the maximum order of a \lf{} in a graph is equal to the maximum order of a \lf{} containing only $P_2$ and $P_3$ components.  We note that sometimes a \lf{} is called a {\em linear forest} in the literature \cite{BR, FS}.

Here we study the Ramsey problem for \lfs{}: what is the order of the largest monochromatic \lf{} we can find in every $r$-coloring of the edges of $K_n$?  Note that this belongs to the part of Ramsey theory where the target graph is a large monochromatic {\em member of a family} instead of a specified graph. Many other families have been investigated, for example the family of connected graphs, graphs without isolated vertices, highly connected graphs, graphs of small diameter, etc. A survey on problems of this flavor is \cite{GYSUR}.

Given positive integers $r, p_1, \dots, p_r$, define $\PM(p_1, \dots, p_r)$ to be the smallest integer $n$ such that in any $r$-coloring of the edges of $K_n$ there exists a \lf{} of color $i$ and order at least $p_i$ for some $i\in [r]$.   If $p_1=\dots=p_r=p$, we sometimes write $\PM_r(p)$ instead of  $\PM(p, \dots, p)$.

Burr and Roberts \cite{BR} proved that for all integers $p\geq 2$, $\PM(p,p)=\ceiling{\frac{4p}{3}}-1$.  Later, Faudree and Schelp \cite{FS} proved a non-symmetric version; that is, for all integers $p_1\geq p_2\geq 2$, $$\PM(p_1,p_2)=p_1+\ceiling{\frac{p_2}{3}}-1.$$
(In fact, in both cases above the authors prove a stronger statement where the formula takes into account the number of paths of odd length.)  We extend these results to $r$-colorings with $r\geq 3$.

Our main result is that we completely determine $\PM$ provided $p_1$ is not too small compared to $r$.  Note the similarity between Theorem \ref{colo} and Theorem \ref{main2}.

\begin{theorem}\label{main2} Let $r\geq 3$ and let $p_1\ge p_2\ge \dots \ge p_r\geq 2$ be integers with $p_1\geq 4$.  If
$p_1\geq 2r-3-\sum_{i=2}^r 3\left(\ceiling{\frac{p_i}{3}}-\frac{p_i}{3} \right),$ then $$\PM(p_1, \dots, p_r)=  p_1- (r-1) +  \sum_{i=2}^{r}\ceiling{\frac{p_i}{3}}.$$
\end{theorem}
The lower bound in Theorem \ref{main2} comes from the extremal coloring
\[
\left[p_1-1,\ceiling{\frac{p_2}{3}}-1,\dots, \ceiling{\frac{p_r}{3}}-1\right].
\]

Note that if at least $r-3$ of the terms $p_2, \dots, p_r$ were congruent to $1\bmod 3$, then $p_1\geq 2r-3-\sum_{i=2}^r 3\left(\ceiling{\frac{p_i}{3}}-\frac{p_i}{3} \right)$ reduces to $p_1\geq 3$ and thus we have an exact result with no extra conditions.

It would be natural to guess that the requirement that $p_1$ be sufficiently large in terms of $r$ is unnecessary. However, in Corollary \ref{corlower}, we will prove that if say $\frac{3}{2}\floor{\frac{\sqrt{8r+1}+1}{2}}>p_1\geq \dots \geq p_r\geq 3$ with all $p_i$ being divisible by 3, then, perhaps surprisingly,
\[
\PM(p_1, \dots, p_r)>p_1-(r-1)+\sum_{i=2}^r\ceiling{\frac{p_i}{3}}.
\]
Thus the complete determination of $\PM(p_1, \dots, p_r)$ is still open and as we will see later, determining $\PM(p_1, \dots, p_r)$ when $p_1$ is fixed and $r$ is large may be difficult because of the connection with covering designs.

However, our next main result shows that in any case we can determine $\PM(p_1, \dots, p_r)$ to within a constant (depending on $r$).

\begin{theorem}\label{main} Let $r\geq 2$ and let $p_1\ge p_2\ge \dots \ge p_r\geq 2$ be integers.
Then $$\PM(p_1, \dots, p_r)\leq \ceiling{p_1-\frac{r}{3}+\sum_{i=2}^r\frac{p_i}{3}}.$$
\end{theorem}

We get the following corollary of Theorem \ref{main} in the case where all the $p_i$-s are equal (stated here using the inverse formulation).

\begin{corollary}\label{diagonal}
Let $r\geq 2$ be an integer. Every $r$-coloring of $K_n$
contains a monochromatic \lf{} of order at least $3\floor{\frac{n}{r+2}}$.
\end{corollary}

This is sharp if $n$ is divisible by $r+2$ as shown by the extremal coloring
$\left[\frac{3n}{r+2}, \frac{n}{r+2}, \dots, \frac{n}{r+2}\right].$

We note that establishing Theorem \ref{main2} requires a bit more technical work than Theorem \ref{main}; however, since Theorem \ref{main2} is tight in many more cases than Theorem \ref{main}, it is worth it.  However, if one was only interested in Corollary \ref{diagonal}, Theorem \ref{main} would suffice.

\section{Covering designs, 1-cores, and a deficiency formula for \lfs{}}\label{tools}

The proof of Theorems \ref{main2} and \ref{main} are based on a minimax theorem on \lfs{} derived from a result of Las Vergnas (which provides an analogue of Tutte's 1-factor theorem for path-matchings).  Interestingly, when we apply this minimax theorem to $r$-colored complete graphs, we need a suitable estimate on block sizes in covering designs (which can be also formulated as an estimate on the sizes of $1$-cores in colored complete graphs).

We describe all of this in detail in the following subsections.

\subsection{Ramsey numbers of covering designs}
A {\em covering design} is a family of sets called {\em blocks} in an $n$-element set $V$ such that each pair of $V$ is covered by {\em at least one block}. If all blocks have the same size $p$, then $C(n,p)$ is used to denote the minimum number of blocks in a covering design. The asymptotics of $C(n,p)$ for fixed $p$ was determined by Erd\H os and Hanani \cite{EH} and  the breakthrough of R. M. Wilson \cite{W} provided equality with constructing block designs for every admissible $n\ge n_0(p)$.

One can formulate the inverse problem of finding $C(n,p)$ as a Ramsey problem. For given $r,p$ find the smallest $n=R_r(p)$ such that every covering design on $n$ vertices with $r$ blocks must contain a block of size at least $p$. Mills \cite{M} determined the asymptotic of $R_r(p)/p$ for $r\le 13$ and this ratio is also known for values of $r$ in the form $q^2+q+1$ or $q^2+q$ when $PG(2,q)$ exists (see the excellent survey of F\"uredi \cite[Chapter 7]{FUSUR}).  This problem was also studied, using a different formulation by Hor\'ak and Sauer \cite{HS}.  However, there is no conjecture for the limit of $R_r(p)/p$ for general $r$.

For our goals we consider covering designs with {\em variable block sizes}, which leads to the off-diagonal case of the Ramsey number $R_r(p)$.
In the next section we will obtain estimates for this Ramsey number.


\subsection{Ramsey numbers of $1$-cores}
The Ramsey number of covering designs can be reformulated in graph theoretic language as the Ramsey number of graphs with minimum degree at least one, i.e. graphs without isolated vertices.  With a slight abuse of the original definition, we say that $G$ is a {\em $1$-core} if $G$ has minimum degree at least one.  (The $k$-core of a graph $G$ was defined by Seidman \cite{S} as the largest {\em connected} subgraph of $G$ with minimum degree at least $k$, subsequently many papers \cite{B} and textbooks \cite{Bo} define it without the connectivity condition.)  To see that the Ramsey number of the family of 1-cores is the same as the Ramsey number of a covering design, given an $r$-coloring of $K_n$, we can replace the 1-core of color $i$ with a clique of color $i$ (allowing for edges to have multiple colors) without changing the size of the 1-core and thus each clique corresponds to a block in the covering design language.

Given positive integers $r, p_1, \dots, p_r$, let $\RC(p_1, \dots, p_r)$ be the smallest integer $n$ such that in every $r$ coloring of the edges of $K_n$, there exists a 1-core of color $i$ and order at least $p_i$ for some $i\in [r]$.  Equivalently, $\RC(p_1, \dots, p_r)$ is the smallest integer $n$ such that the edges of $K_n$ cannot be covered with cliques of order $p_1-1, \dots, p_r-1$.  If $p_1=\dots=p_r=p$, we write $\RC_r(p)$ instead of  $\RC(p, \dots, p)$.  We also note that Observation \ref{p_i=2} and Proposition \ref{all3} apply with $\RC$ in place of $\PM$.

First note the following which essentially means that we can assume $p_i\geq 3$ for all $i\in [r]$.

\begin{observation}\label{p_i=2}
For all integers $r\geq 2$ and $p_1\geq p_2\geq \dots \geq p_r\geq 2$,
$\PM(p_1, \dots, p_r)=\PM(p_1, \dots, p_r, 2)=\PM(p_1, \dots, p_r, 1)$ and similarly for $R^{1C}$.
\end{observation}

Next we have the following result when $p_1=\dots=p_r=3$.

\begin{proposition}\label{all3}
Let $r\geq 2$ be an integer.  Then $\PM_r(3)=\RC_r(3)$ is the smallest integer $n$ such that $\binom{n}{2}>r$.  In other words, $\PM_r(3)=\RC_r(3)=\floor{\frac{\sqrt{8r+1}+1}{2}}+1$.
\end{proposition}

\begin{proof}
If $\binom{n}{2}>r$, then in every $r$-coloring of $K_n$, some color must be used more than once.  If $\binom{n}{2}\leq r$, then there exists an $r$-coloring of $K_n$ where each color is used at most once.
\end{proof}

We now give estimates on $\RC(p_1, \dots, p_r)$ where the focus is on the off-diagonal case, which to the best of our knowledge has not been studied.

In the language of covering designs, we have $\RC_r(p)\leq n$ if and only if $C(n,p-1)>r$.  So the vast literature on covering designs gives upper bounds on $\RC_r(p)$.  A simple lower bound given by Erd\H{o}s and Hanani \cite{EH} is $C(v, p-1)\geq \binom{v}{2}/\binom{p-1}{2}=\frac{v(v-1)}{(p-1)(p-2)}$.  A more refined lower bound is the so-called Sch\"onheim bound \cite{Sch} , 
which says $C(v, p-1)\geq \ceiling{\frac{v}{p-1}\ceiling{\frac{v-1}{p-2}}}$.

This first Observation is just a generalization of the Erd\H{o}s-Hanani lower bound in the non-uniform case.  It says that if the total number of edges in cliques of orders $p_1-1, \dots, p_r-1$ respectively is less than the number of edges in $K_{n}$, then it is not possible to cover $K_n$ with those cliques.  While this Observation is almost trivial, there are situations where this is the best estimate to use.

\begin{observation}\label{edgecount}
Let $n,r\geq 2$ and let $p_1\ge p_2\ge \dots \ge p_r\geq 2$ be integers.  If
\begin{equation}
\sum_{i=1}^r\binom{p_i-1}{2}<\binom{n}{2},
\end{equation}
then $\RC(p_1, \dots, p_r)\leq n$.
\end{observation}

The next bound is essentially a non-uniform generalization of results in \cite{HS} and \cite{M}.  Again there are situations where this next bound is the best estimate to use.  This result can be seen to be tight when a projective plane of order $q$ exists, $\ell\in \{0,1\}$ and $p_1=\dots=p_r=\ceiling{\frac{(q+1)n}{q^2+q+\ell}}$.

\begin{proposition}\label{1-corehalf} Let $n,r\geq 2$ and let $p_1\ge p_2\ge \dots \ge p_r\geq 2$ be integers.  If there exists an integer $1\leq t\leq r-1$ such that
\begin{equation}
p_1\leq \ceiling{\frac{n+t-1}{t}} ~\text{ and }~ \sum_{i=1}^r(p_i-1)<(t+1)n,
\end{equation}
then $\RC(p_1, \dots, p_r)\leq n$.
\end{proposition}

We say that a vertex \emph{sees a color} if it is incident with an edge of that color.

\begin{proof}
Suppose there exists an integer $1\leq t\leq r-1$ such that $p_1\leq \ceiling{\frac{n+t-1}{t}}$ and $\sum_{i=1}^r(p_i-1)<(t+1)n$ and consider an $r$-edge coloring of a complete graph on a set $V$ of vertices where $|V|=n$.  For all $i\in [r]$, let $S_i$ be the 1-core of color $i$ and  suppose for contradiction that $|S_i|\leq p_i-1$ for all $i\in [r]$.  So for all $i\in [r]$, there exists at least $n-p_i+1$ vertices which do not see color $i$.

So on average, the number of colors a vertex does not see is at least $$\frac{\sum_{i=1}^r (n-p_i+1)}{n}=r-\frac{\sum_{i=1}^r (p_i-1)}{n}>r-(t+1).$$
This implies some vertex $v$ sees at most $t$ colors.  So $v$ is contained in a monochromatic 1-core of order at least $1+\ceiling{\frac{n-1}{t}}=\ceiling{\frac{n+t-1}{t}}$ contradicting the original assumption.
\end{proof}

Now we come to our main result of this section.  We note that this result can be seen to be tight in certain cases, such as when either of the first two terms are the maximum.

\begin{theorem}\label{1-core} Let $r\geq 2$ and let $p_1\ge p_2\ge \dots \ge p_r\geq 2$ be integers.  Then
$$\RC(p_1, \dots, p_r)\leq \max\left\{p_1,\ \ceiling{\frac{p_1+p_2+p_3}{2}}-1,\ \ceiling{\frac{p_1}{3}-\frac{r}{3} + \sum_{i=1}^{r}\frac{p_i}{3}}\right\}.$$
\end{theorem}

\begin{proof}
Consider an $r$-edge coloring of a complete graph on a set $V$ of vertices where
\begin{equation}\label{V}
|V|=n=
\max\left\{p_1,\ \ceiling{\frac{p_1+p_2+p_3}{2}}-1,\ \ceiling{\frac{p_1}{3}-\frac{r}{3} + \sum_{i=1}^{r}\frac{p_i}{3}}\right\}.
\end{equation}

For all $i\in [r]$, let $S_i$ be the 1-core of color $i$ and suppose for contradiction that $|S_i|\leq p_i-1$ for all $i\in [r]$.

Consider the partition $V=X_1\cup X_2\cup X_3$, where for $i\in [2]$, $X_i$ is the set of vertices which are incident with edges of exactly $i$ different colors and $X_3$ is the set of vertices incident with edges of at least $3$ different colors.  Note that every vertex is incident with edges of at least two different colors (i.e. $X_1=\emptyset$); otherwise there is a monochromatic 1-core on $n$ vertices, but by \eqref{V} and the indirect assumption, we have $p_1\leq n\leq p_1-1$, a contradiction.

Note that
\beq\label{sieve}
|V|=n \leq \sum_{i=1}^{r} |S_i| -|X_2|-2|X_3|,
\eeq
since the vertices in $X_2$ are counted twice in the sum $\sum_{i=1}^{r} |S_i|$, and the vertices in $X_3$ are counted at least three times in the sum $\sum_{i=1}^{r} |S_i|$.

\begin{claim}\label{claim}
$|X_2| \leq p_1-1$.
\end{claim}

\begin{proof}
Since the graph induced by $X_2$ is locally 2-colored (i.e. each vertex sees at most two colors), either there exists a color $i$ such that every vertex in $X_2$ sees color $i$, in which case by the indirect assumption, $|X_2|\leq p_i-1\leq p_1-1$, or there are a total of at most three colors used on $X_2$.  We now show that the latter is impossible (because of the choice of $n$).

Suppose there is no color seen by every vertex in $X_2$ and let $i, j, k$ be the three colors used on $X_2$.  In this case, it can be easily seen that $X_2$ contains a set of three vertices $\{a,b,c\}$ such that $ab$ has color $i$, $ac$ has color $j$ and $bc$ has color $k$.  Since every edge incident with $a$ has color $i$ or $j$, every edge incident with $b$ has color $i$ or $k$, and every edge incident with $c$ has color $j$ or $k$, it is the case that for all $v\in V\setminus \{a,b,c\}$, $v$ sends at least two of the colors $i,j,k$ to $\{a,b,c\}$.  So, regardless of whether $X_3=\emptyset$ or not, every vertex in $V$ sees at least two of the colors $i,j,k$ and thus $$2n\leq |S_i|+|S_j|+|S_k|\leq |S_1|+|S_2|+|S_3|\leq p_1+p_2+p_3-3,$$
contradicting the choice of $n$.
\end{proof}

Now by Claim \ref{claim} we have
\beq\label{x3}
|X_3|=n-|X_2| \geq n-p_1+1.
\eeq

Then from \eqref{sieve}  and \eqref{x3} we get
\begin{align*}
n \leq \sum_{i=1}^{r} |S_i| -|X_2|-2|X_3|&\leq \sum_{i=1}^{r} p_i - n-|X_3|-r\leq \sum_{i=1}^{r} p_i - n-(n-p_1+1)-r
\end{align*}
which, by \eqref{V}, implies $$3\ceiling{\frac{p_1}{3}-\frac{r}{3} + \sum_{i=1}^{r} \frac{p_i}{3}}\leq 3n\leq p_1-(r+1)+\sum_{i=1}^r p_i,$$
a contradiction.
\end{proof}

\section{Large monochromatic \lfs{}}

\subsection{Deficiency formula for \lfs{}}

The deficiency formula for \lfs{} can be derived from a special case of a result of Las Vergnas \cite{LV}.  Recall that a \lf{} can always be written as vertex disjoint union of $P_2$-s and $P_3$-s, so the maximum order of a \lf{} in a graph is equal to the maximum order of a \lf{} containing only $P_2$ and $P_3$ components. Let $f,g$ be integer-valued functions on the vertex set $V$ of a graph $G$ such that $0\le g(v)\le 1\le f(v)$ for all $v\in V$.  A {\em $(g,f)$-factor} is a subgraph $F$ of $G$ satisfying $g(v)\le d_F(v)\le f(v)$ for all $v\in V$. Las Vergnas \cite{LV} gave a necessary and sufficient condition for the existence of a $(g,f)$ factor of a graph. If $g\equiv 1,f\equiv 2$ then the existence of a $(g,f)$-factor is equivalent to the existence of a  {\em perfect \lf{}}; that is, a \lf{} covering all vertices of $G$. In this case the condition simplifies and can be stated as follows.  Let $q_G(S)$ denote the number of isolated vertices of a graph $G$ in a set $S\subset V(G)$.

\begin{theorem}[Las Vergnas \cite{LV}]\label{lvth}There exists a perfect \lf{} in $G$ if and only if $2|X|\ge q_G(V(G)\setminus X)$ for all $X\subset V(G)$.
\end{theorem}

This result is ``self-refining'' in the sense that one can easily derive from it the minimax formula for the deficiency of \lfs{} (see \cite[Exercise 3.1.16]{LP} in which Berge's formula is derived from Tutte's theorem). Let $pd(G)$ be the {\em \lf{} deficiency} of $G$, the number of vertices uncovered by any  \lf{} of maximum order in $G$.

\begin{corollary}\label{lvthdef} $pd(G)=max\{q_G(V(G)\setminus X)-2|X|: X\subset V(G)\}$.
\end{corollary}

We call a set $X$ achieving the maximum in Corollary \ref{lvthdef} an {\em LV set}.


\subsection{1-cores and \lfs{}}

Our main general result shows that the Ramsey numbers for \lfs{} are tied to the Ramsey numbers for 1-cores in a fundamental way.  First, given positive integers $r\geq 2$, $p_1\ge p_2\ge \dots \ge p_r\geq 2$, and $d$, let
\begin{equation}\label{f}
f^d(p_1, \dots, p_r)=\max\{\RC(p_1-dx_1, \dots, p_r-dx_r)+\sum_{i=1}^r x_i:  0\leq x_i< \frac{p_i}{d} \text{ for all } i\in [r]\},
\end{equation}
where the $x_i$'s are integers.

\begin{theorem}\label{mainthm}
Let $r\geq 2$ and let $p_1\ge p_2\ge \dots \ge p_r\geq 2$ be integers.  Then
$$\PM(p_1, \dots, p_r)=f^3(p_1, \dots, p_r).$$
\end{theorem}

For instance this result says $\PM(6,6,6,6,6)=\max\{\RC(6,6,6,6,6), \RC(6,6,6,6,3)+1, \RC(6,6,6,3,3)+2, \RC(6,6,3,3,3)+3, \RC(6,3,3,3,3)+4, \RC(3,3,3,3,3)+5\}$.

The following example provides the lower bound in Theorem \ref{mainthm}.

\begin{example}\label{mainexample}
Let $r, p_1, \dots, p_r, x_1, \dots, x_r$ be integers with $r\geq 2$, $p_1\ge p_2\ge \dots \ge p_r\geq 2$, and for all $i\in [r]$, $0\leq x_i< \frac{p_i}{3}$.  If $n=\RC(p_1-3x_1, \dots, p_r-3x_r)+\sum_{i=1}^rx_i$, then there exists an $r$-coloring of $K_{n-1}$ such that for all $i\in [r]$, the largest \lf{} of color $i$ has order at most $p_i-1$.
\end{example}

\begin{proof}
Set $t=\RC(p_1-3x_1, \dots, p_r-3x_r)$ and start with an $r$-coloring of $K_{t-1}$ such that for all $i\in [r]$, the largest 1-core of color $i$ has order at most $p_i-3x_i-1$ which must exist by the definition of $\RC(p_1-3x_1, \dots, p_r-3x_r)$.  For all $i\in [r]$, add a set $X_i$ of vertices such that $|X_i|=x_i$ and color all edges incident with $X_i$ with color $i$.  This gives a coloring of $K_{n-1}$ such that for all $i\in [r]$, the largest \lf{} of color $i$ has order at most $p_i-3x_i-1+3x_i=p_i-1$ (by Corollary \ref{lvthdef} for instance).
\end{proof}

The following proof is inspired by Petrov's \cite{P} non-inductive proof of Theorem \ref{colo} (for another similar proof see \cite{XYZ}).

\begin{proof}[Proof of Theorem \ref{mainthm}]
The lower bound follows from Example \ref{mainexample}.

For the upper bound, consider an $r$-edge coloring of a complete graph on a vertex set $V$ with $|V|=n=f^3(p_1, \dots, p_r)$.

For all $i\in [r]$, let $G_i$ the subgraph induced by the edges of color $i$.  Suppose, for contradiction, that for all $i\in [r]$ the largest \lf{} in $G_i$ has order at most $p_i-1$.  For all $i\in [r]$, apply Corollary \ref{lvthdef} to get an LV set $X_i$ and corresponding independent set $S_i$ such that
\begin{equation}
|S_i|\geq 2|X_i|+n-(p_i-1),
\end{equation}
and note that since $|S_i|\leq n-|X_i|$, we have $$0\leq |X_i|<\frac{p_i}{3}.$$  Let $X=\cup_{i=1}^r X_i$ and let $Y=V\setminus X$.  For all $i\in [r]$ we have $$|S_i\cap Y|\geq |S_i|-(|X|-|X_i|)\geq 2|X_i|+|V|-(p_i-1)-(|X|-|X_i|)=|Y|-(p_i-3|X_i|-1).$$  In other words, the largest 1-core of color $i$ in the graph $G_i[Y]$ has order at most $p_i-3|X_i|-1$.

We have by the definition of $n$,
\begin{align*}
|Y|=n-|X|\geq n-\sum_{i=1}^r|X_i|\geq \RC(p_1-3|X_1|, \dots, p_r-3|X_r|)
\end{align*}
so $Y$ contains a 1-core of color $i$ and order at least $p_i-3|X_i|$ for some $i\in [r]$, contradicting our original assumption.
\end{proof}

Now we obtain lower bounds on $\PM(p_1, \dots, p_r)$ as a corollary of Theorem \ref{mainthm}.

\begin{corollary}\label{corlower}
Let $r, p_1, \dots, p_r$ be integers with $r\geq 2$, $p_1\ge p_2\ge \dots \ge p_r\geq 2$ and  let $s$ be the number of $p_i$'s which are divisible by 3.
Then the following hold:
\begin{enumerate}
\item $\PM(p_1, \dots, p_r)\geq p_1-(r-1)+\sum_{i=2}^r\ceiling{\frac{p_i}{3}}$,

\item $\PM(p_1, \dots, p_r)\geq \RC_s(3)+\sum_{i=1}^r(\ceiling{\frac{p_i}{3}}-1)=\floor{\frac{\sqrt{8s+1}+1}{2}}+1+\sum_{i=1}^r(\ceiling{\frac{p_i}{3}}-1)$.
\end{enumerate}
\end{corollary}

Note that under certain circumstances, such as all $p_i$ being divisible by 3 and $p_1<\frac{3}{2}\floor{\frac{\sqrt{8r+1}+1}{2}}$, we have $$\floor{\frac{\sqrt{8r+1}+1}{2}}+1+\sum_{i=1}^r(\ceiling{\frac{p_i}{3}}-1)>p_1-(r-1)+\sum_{i=2}^r\ceiling{\frac{p_i}{3}}.$$
For instance, we have $\PM_{10}(6)>6-(10-1)+(10-1)\cdot 2=15$.

\begin{proof}
(i) This follows from the fact that
\begin{align*}
&\max\{\RC(p_1-3x_1, \dots, p_r-3x_r)+\sum_{i=1}^r x_i:  0\leq x_i< \frac{p_i}{3} \text{ for all } i\in [r]\}\\
&\geq \RC(p_1, p_2-3(\ceiling{\frac{p_2}{3}}-1), \dots, p_r-3(\ceiling{\frac{p_r}{3}}-1))+\sum_{i=2}^r(\ceiling{\frac{p_i}{3}}-1)\geq p_1+\sum_{i=2}^r(\ceiling{\frac{p_i}{3}}-1),
\end{align*}
where the last inequality holds since in general we have $\RC(a_1, \dots, a_r)\geq a_1$ for all $a_1\geq \dots\geq a_r$.

(ii) This follows from the fact that
\begin{align*}
&\max\{\RC(p_1-3x_1, \dots, p_r-3x_r)+\sum_{i=1}^r x_i:  0\leq x_i< \frac{p_i}{3} \text{ for all } i\in [r]\}\\
&\geq \RC(p_1-3(\ceiling{\frac{p_1}{3}}-1), p_2-3(\ceiling{\frac{p_2}{3}}-1), \dots, p_r-3(\ceiling{\frac{p_r}{3}}-1))+\sum_{i=1}^r(\ceiling{\frac{p_i}{3}}-1)\\
&=\RC_s(3)+\sum_{i=1}^r(\ceiling{\frac{p_i}{3}}-1)= \floor{\frac{\sqrt{8s+1}+1}{2}}+1+\sum_{i=1}^r(\ceiling{\frac{p_i}{3}}-1)
\end{align*}
where the last equality holds by Proposition \ref{all3}.
\end{proof}

Now we show how Theorem \ref{mainthm} implies the result for two colors (which follows from the result of Faudree and Schelp \cite{FS}). However, the proof of their result is lengthy and relies on the $2$-color Ramsey number of paths determined in \cite{GGY}.  Therefore the short proof below is perhaps of some interest.

Note that since a graph or its complement is connected, we clearly have
\begin{equation}\label{rc2}
\RC(p_1, p_2)=\max\{p_1, p_2\}.
\end{equation}

\begin{corollary}\label{r=2}
Let $p_1\ge p_2\geq 2$ be integers.  Then
$$\PM(p_1, p_2)=p_1+\ceiling{\frac{p_2}{3}}-1.$$
\end{corollary}

\begin{proof}
By \eqref{rc2}, we have for all $0\leq x_1< \frac{p_1}{3}, 0\leq x_2< \frac{p_2}{3}$,
\[
\RC(p_1-3x_1, p_2-3x_2)+x_1+x_2\leq \max\{p_1-3x_1, p_2-3x_2\}+x_1+x_2\leq p_1+x_2\leq p_1+\ceiling{\frac{p_2}{3}}-1.
\]
Thus by Theorem \ref{mainthm} we have $\PM(p_1, p_2)=f^3(p_1, p_2)\leq p_1+\ceiling{\frac{p_2}{3}}-1$.
\end{proof}

\subsection{The proofs of Theorem \ref{main2} and Theorem \ref{main}}

We obtain the proofs by combining using Theorem \ref{1-core} to get an appropriate upper bound on $f^3(p_1, \dots, p_r)$ at which point the result follows from Theorem \ref{mainthm}. So the proof just reduces to checking some technical inequalities which we first collect here.

\begin{fact}\label{techfact}
Let $r\geq 3$ and let $a_1\ge a_2\ge \dots \ge a_r\geq 2$ be integers with $a_1\geq 3$.
\begin{enumerate}
\item $\ceiling{\frac{2a_1}{3}-\frac{r}{3}+\sum_{i=1}^r\frac{a_i}{3}}\geq \ceiling{\frac{a_1+a_2+a_3}{2}}-1$

\item $a_1- (r-1) +  \sum_{i=2}^{r}\ceiling{\frac{a_i}{3}}\geq \ceiling{\frac{a_1}{3}-\frac{r}{3} + \sum_{i=1}^{r}\frac{a_i}{3}}$ if and only if  $a_1\geq 2r-3-\sum_{i=2}^r 3\left(\ceiling{\frac{a_i}{3}}-\frac{a_i}{3} \right).$

\item $a_1- (r-1) +  \sum_{i=2}^{r}\ceiling{\frac{a_i}{3}}\geq \ceiling{\frac{a_1+a_2+a_3}{2}}-1+  \sum_{i=4}^{r}(\ceiling{\frac{a_i}{3}}-1)$ if and only if $a_1\geq 2+\left(a_2-2\ceiling{\frac{a_2}{3}}\right)+\left(a_3-2\ceiling{\frac{a_3}{3}}\right).$  In particular, this holds when $a_1\geq 4$.
\end{enumerate}
\end{fact}

\begin{proof}
(i) First note that $\ceiling{\frac{2a_1}{3}-\frac{r}{3}+\sum_{i=1}^r\frac{a_i}{3}}\geq a_1+\frac{a_2}{3}+\frac{a_3}{3}-1$ and since $a_1\geq 3$, we have $a_1+\frac{a_2}{3}+\frac{a_3}{3}-1\geq \frac{a_1+a_2+a_3-1}{2}\geq \ceiling{\frac{a_1+a_2+a_3}{2}}-1$

(ii) Since $a_1- (r-1) +  \sum_{i=2}^{r}\ceiling{\frac{a_i}{3}}$ is an integer, $a_1- (r-1) +  \sum_{i=2}^{r}\ceiling{\frac{a_i}{3}}\geq \ceiling{\frac{a_1}{3}-\frac{r}{3} + \sum_{i=1}^{r}\frac{a_i}{3}}$ is equivalent to $a_1-(r-1) + \sum_{i=2}^{r}\ceiling{\frac{a_i}{3}}\geq \frac{a_1}{3}-\frac{r}{3} + \sum_{i=1}^{r}\frac{a_i}{3}$ which holds precisely when $a_1\geq  2r-3-\sum_{i=2}^r 3(\ceiling{\frac{a_i}{3}}-\frac{a_i}{3})$.

(iii) Since $a_1- (r-1) +  \sum_{i=2}^{r}\ceiling{\frac{a_i}{3}}$ is an integer, $a_1- (r-1) +  \sum_{i=2}^{r}\ceiling{\frac{a_i}{r}}\geq \ceiling{\frac{a_1+a_2+a_3}{2}}-1+  \sum_{i=4}^{r}(\ceiling{\frac{a_i}{3}}-1)$ is equivalent to $a_1+\ceiling{\frac{a_2}{3}}+\ceiling{\frac{a_2}{3}}\geq \frac{a_1+a_2+a_3}{2}+1$ which holds precisely when $a_1\geq 2+\left(a_2-2\ceiling{\frac{a_2}{3}}\right)+\left(a_3-2\ceiling{\frac{a_3}{3}}\right)$.

Note that if $a_1\geq 4$, we have $a_1\geq 2+2\left(a_1-2\ceiling{\frac{a_1}{3}}\right)\geq 2+\left(a_2-2\ceiling{\frac{a_2}{3}}\right)+\left(a_3-2\ceiling{\frac{a_3}{3}}\right)$ where the last inequality holds since $a_1\geq a_2\geq a_3$.
\end{proof}

\begin{proof}[Proof of Theorem \ref{main2} and Theorem \ref{main}]
Let $r\geq 3$ and let $p_1\geq \dots \geq p_r\geq 2$ (note that we already dealt with the case when $r=2$ in Corollary \ref{r=2}).  Let $x_1, \dots, x_r$ be integers with $0\leq x_i<\frac{p_i}{3}$ for all $i\in [r]$.  We have $p_1\geq \dots\geq p_r$, but it might not be the case that $p_1-3x_1\geq \dots\geq p_r-3x_r$, so we let $q_1\geq \dots \geq q_r$ be integers and we let $\pi$ be a permutation on $[r]$ such that $p_{\pi(i)}-3x_{\pi(i)}=q_{i}$ for all $i\in [r]$.

To prove Theorem \ref{main}, we note that by Theorem \ref{1-core} we have
\begin{align*}
\RC(p_1-3x_1, \dots, p_r-3x_r)&\leq \max\left\{q_1,\ \ceiling{\frac{q_1+q_2+q_3}{2}}-1,\ \ceiling{\frac{q_1}{3}-\frac{r}{3} + \sum_{i=1}^{r}\frac{q_i}{3}}\right\}\\
&\leq  \ceiling{\frac{2q_1}{3}-\frac{r}{3}+\sum_{i=1}^r\frac{q_i}{3}}=\ceiling{\frac{2p_{\pi(1)}}{3}-2x_{\pi(1)}-\frac{r}{3}+\sum_{i=1}^r(\frac{p_i}{3}-x_i)}.
\end{align*}
where the second inequality is by applying Fact \ref{techfact}.(i) (with $a_i=q_i$) to the second term in the maximum (the inequality is trivial for the other two terms in the maximum).

So we have
\[
\RC(p_1-3x_1, \dots, p_r-3x_r)+\sum_{i=1}^rx_i\leq \ceiling{\frac{2p_{\pi(1)}}{3}-2x_{\pi(1)}-\frac{r}{3}+\sum_{i=1}^r\frac{p_i}{3}}\leq \ceiling{ p_1-\frac{r}{3} + \sum_{i=2}^{r} \frac{p_i}{3}},
\]
and thus by Theorem \ref{mainthm},  $\PM(p_1, \dots, p_r)=f^3(p_1, \dots, p_r)\leq \ceiling{ p_1-\frac{r}{3} + \sum_{i=2}^{r} \frac{p_i}{3}}$ as desired.

Next we prove Theorem \ref{main2}. The lower bound follows from Corollary \ref{corlower}.(i).
For the upper bound, we set $m:=\max\left\{q_1,\ \ceiling{\frac{q_1+q_2+q_3}{2}}-1,\ \ceiling{\frac{q_1}{3}-\frac{r}{3} + \sum_{i=1}^{r}\frac{q_i}{3}}\right\}$ and note that by Theorem \ref{1-core}, we have
\begin{align*}
&\RC(p_1-3x_1, \dots, p_r-3x_r)+\sum_{i=1}^rx_i\leq m+\sum_{i=1}^rx_i\\
&\stackrel{*}{\leq} \max\left\{p_{1}-(r-1)+\sum_{i=2}^r\ceiling{\frac{p_i}{3}},\ \ceiling{\frac{p_1+p_2+p_3}{2}}-1+\sum_{i=4}^{r}(\ceiling{\frac{p_i}{3}}-1),\ \ceiling{\frac{p_1}{3}-\frac{r}{3} + \sum_{i=1}^{r}\frac{p_i}{3}}\right\}\\
&\leq p_{1}-(r-1)+\sum_{i=2}^r\ceiling{\frac{p_i}{3}},
\end{align*}
where the last inequality follows by Fact \ref{techfact}.(ii) and Fact \ref{techfact}.(iii) (with $a_i=p_i$).  It remains to justify the inequality $\stackrel{*}{\leq}$ above.

\textbf{Case 1} ($m = q_1$)  Set $J=[r]\setminus \{\pi(1)\}$.  We have
\begin{align*}
\RC(p_1-3x_1, \dots, p_r-3x_r)+\sum_{i=1}^rx_i
\leq p_{\pi(1)}-2x_{\pi(1)}+\sum_{j\in J}x_j
&\leq p_{\pi(1)}-2x_{\pi(1)}+\sum_{j\in J}(\ceiling{\frac{p_j}{3}}-1)\\
&\leq p_{1}-(r-1)+\sum_{i=2}^r\ceiling{\frac{p_i}{3}}.
\end{align*}

\textbf{Case 2} ($m = \ceiling{\frac{q_1+q_2+q_3}{2}}-1$)  Set $J=[r]\setminus \{\pi(1), \pi(2), \pi(3)\}$.  We have
\begin{align*}
\RC(p_1-3x_1, \dots, p_r-3x_r)+\sum_{i=1}^rx_i
&\leq \ceiling{\frac{p_{\pi(1)}-x_{\pi(1)}+p_{\pi(2)}-x_{\pi(2)}+p_{\pi(3)}-x_{\pi(3)}}{2}}-1+\sum_{j\in J}x_j\\
&\leq \ceiling{\frac{p_{\pi(1)}+p_{\pi(2)}+p_{\pi(3)}}{2}}-1+\sum_{j\in J}(\ceiling{\frac{p_j}{3}}-1)\\
&\leq \ceiling{\frac{p_{1}+p_{2}+p_{3}}{2}}-1+\sum_{i=4}^r(\ceiling{\frac{p_i}{3}}-1).
\end{align*}

\textbf{Case 3} ($m=\ceiling{\frac{q_1}{3}-\frac{r}{3} + \sum_{i=1}^{r}\frac{q_i}{3}}$)  We have
\begin{align*}
\RC(p_1-3x_1, \dots, p_r-3x_r)+\sum_{i=1}^rx_i
\leq \ceiling{\frac{p_{\pi(1)}}{3}-x_{\pi(1)}-\frac{r}{3}+\sum_{i=1}^r\frac{p_i}{3}}\leq \ceiling{\frac{p_{1}}{3}-\frac{r}{3}+\sum_{i=1}^r\frac{p_i}{3}}.
\end{align*}
\end{proof}

\subsection{Small values of $p_i$ and small number of colors}

We now determine the Ramsey number of \lfs{} for at most 4 colors.

\begin{corollary}\label{r234}
~
\begin{enumerate}
\item[(i)] For all integers $p_1\geq p_2\geq 2$, $\PM(p_1,p_2)=p_1+\ceiling{\frac{p_2}{3}}-1$.

\item[(ii)] For all integers $p_1\geq p_2\geq p_3\geq 2$ such that $(p_1, p_2, p_3)\neq (3,3,3)$, $\PM(p_1, p_2, p_3)= p_1+\ceiling{\frac{p_2}{3}}+\ceiling{\frac{p_3}{3}}-2$.  Furthermore $\PM(3,3,3)=4$.

\item[(iii)] For all integers $p_1\geq p_2\geq p_3\geq p_4\geq 2$ such that $(p_1, p_2, p_3, p_4)\not\in\{(3,3,3,3), (4,3,3,3)\}$, $\PM(p_1, p_2, p_3, p_4)= p_1+\ceiling{\frac{p_2}{3}}+\ceiling{\frac{p_3}{3}}+\ceiling{\frac{p_4}{3}}-3$.  Furthermore $\PM(3,3,3,3)=4$ and $\PM(4,3,3,3)=5$.
\end{enumerate}
\end{corollary}

\begin{proof} Note that by Observation \ref{p_i=2}, we may assume that $p_1\geq p_2\geq p_3\geq p_4\geq 3$.

(i), (ii)  If $p_1\geq 4$, we may apply Theorem \ref{main2}.  Otherwise $p_1\leq 3$ in which case we are done by Observation \ref{all3}.

(iii)  If $p_1\geq 5=2\cdot 4-3$, then we may apply Theorem \ref{main2}.  So suppose $p_1\leq 4$.  If $(p_1, p_2, p_3, p_4)\in \{(4,4,4,4), (4,4,4,3), (4,4,3,3)\}$, then we may still apply Theorem \ref{main2}.  If $p_1=p_2=p_3=p_4=3$, then we may apply Observation \ref{all3}. So we only have to check the case where $(p_1, p_2, p_3, p_4)=(4,3,3,3)$ which is easily seen by direct inspection.
\end{proof}

Scobee \cite{SCO} determined the Ramsey number $R(m_1P_3,m_2P_3,m_3P_3)=m_2+m_3+3m_1-2$ for $m_1\ge m_2\ge m_3$ and $m_1\ge 2$ with a difficult (20 page) proof.  This implies Corollary \ref{r234}(ii) when all $p_i$'s are divisible by $3$.

We now prove a result about $\PM_r(p)$ when $p=4,5$ (we already determined $\PM_r(3)$ in Observation \ref{all3}).  To do so, we first establish the following.

\begin{proposition}\label{1coresmall}For all integers $r\geq 2$, $\RC_r(4)\leq r+3$, $\RC_r(5)\leq r+4$, and $\RC_r(6)\leq r+5$.
\end{proposition}

Recall that in the language of covering designs, we have $\RC_r(p)\leq v$ if and only if $C(v, p-1)>r$.

\begin{proof}
We have $C(r+3, 3)\geq \frac{(r+3)(r+2)}{3\cdot 2}>r$ for all $r$, and thus $\RC_r(4)\leq r+3$.  We have $C(r+4, 4)\geq \frac{(r+4)(r+3)}{4\cdot 3}>r$ for all $r$, and thus $\RC_r(5)\leq r+4$. Finally, we have $C(r+5, 5)\geq \ceiling{\frac{r+4}{4}\ceiling{\frac{r+3}{3}}}>r$ for all $r$ such that $r\not\in \{4,8\}$. However, in the case when $r=4$ and $r=8$, it is known (see \cite{M} for instance) that $C(9,5)=5>4$ and $C(13,5)=10>8$.  Thus $\RC_r(6)\leq r+5$.
\end{proof}

%

\begin{proposition}\label{PM45}
For all $r\geq 2$, $\PM_r(4)=r+3$ and $\PM_r(5)=r+4$.
\end{proposition}

\begin{proof}
By Proposition \ref{1coresmall} and Theorem \ref{mainthm}, we have
\[
\PM_r(4)=\max\{\RC_s(4)+r-s: s\in [r]\}=\max\{s+3+r-s: s\in [r]\}=r+3$$ and $$\PM_r(5)=\max\{\RC_s(5)+r-s: s\in [r]\}=\max\{s+4+r-s: s\in [r]\}=r+4.\qedhere
\]
\end{proof}
%
%

\section{Conclusion}\label{conclusion}

We have determined $\PM(p_1, \dots, p_r)$ exactly unless $p_1\leq 2r-4-\sum_{i=2}^r 3\left(\ceiling{\frac{p_i}{3}}-\frac{p_i}{3} \right)$ and we have determined $\PM(p_1, \dots, p_r)$ to within a constant (depending on $r$) in every case.  It would certainly be interesting to solve these remaining cases exactly, although Corollary \ref{corlower} gives some evidence that this remaining case where $p_1$ is fixed and $r$ is large could be challenging.

In particular, we think the following problem is worth focusing on.

\begin{problem}
For all integers $r\geq 5$ and $p\geq 6$ such that $p\equiv 0\bmod 3$ or $p\equiv 2\bmod 3$, determine $\PM_r(p)$ exactly.
\end{problem}

Note that the case when $r\leq 4$ follows from Observation \ref{r234}.  The case when $p=2$ is trivial, the case $p=3$ is Observation \ref{all3}, the cases $p=4,5$ follow from Proposition \ref{PM45}, and the case $p\equiv 1\bmod 3$ follows from Theorem \ref{main2}.  Also note that by Theorem \ref{main2} we may assume $r\geq \frac{p+4}{2}$ when $p\equiv 0\bmod 3$ and we may assume $r\geq p+2$ when $p\equiv 2\bmod 3$.

It already seems non-trivial to determine $\PM_r(6)$ for all $r\geq 5$, because in order to do so, one needs to compute (or appropriately bound) $\RC(p_1, \dots, p_a, p_{a+1}, \dots, p_r)$ where $p_1=\dots=p_a=6$ and $p_{a+1}=\dots=p_r=3$ for all $0\leq a\leq r$.

Another question relates to our estimates on $\RC(p_1, \dots, p_r)$.  We think of $p_1-(r-1)+\sum_{i=2}^r\ceiling{\frac{p_i}{3}}$ as being the ``standard'' value of $\PM(p_1, \dots, p_r)$, but of course we know there are examples where $\PM(p_1, \dots, p_r)>p_1-(r-1)+\sum_{i=2}^r\ceiling{\frac{p_i}{3}}$. However, all such examples stem from the fact that $\RC_s(3)=\floor{\frac{\sqrt{8s+1}+1}{2}}+1$ (see Corollary \ref{corlower}).  In every other case we know of, we have $\RC(p_1, \dots, p_r)\leq p_1-(r-1)+\sum_{i=2}^r\ceiling{\frac{p_i}{3}}$ (and in many cases it is much smaller).  So this raises the following question.

\begin{question}
Under what circumstances do we have $\RC(p_1, \dots, p_r)\leq p_1-(r-1)+\sum_{i=2}^r\ceiling{\frac{p_i}{3}}$?  Is this always true when none of $p_1\geq  \dots\geq p_r$ are divisible by 3?   Is this always true when $p_1\geq  \dots\geq  p_r\geq 4$?
\end{question}
One case in which we are confident the above is true is when $p=p_1=\dots=p_r\geq 4$ (and we proved it for $4\leq p\leq 6$ in Proposition \ref{1coresmall}).

Another reason it seems to be difficult to determine $\PM(p_1, \dots, p_r)$ exactly in all cases is that the lower bound examples are not necessarily unique.  We have multiple examples where $\PM(p_1, \dots, p_r)=\RC(p_1, \dots, p_r)$ and in those cases the lower bound for $\PM(p_1, \dots, p_r)$ comes from the extremal coloring $[p_1-1, \ceiling{\frac{p_2}{3}}-1, \dots, \ceiling{\frac{p_r}{3}}-1]$ whereas the lower bound for $\RC(p_1, \dots, p_r)$ comes from a covering design.  For instance, we have $\PM(5,5,5)=7=\RC(5,5,5)$. The extremal coloring $[4,1,1]$ provides a lower bound for $\PM(5,5,5)$, whereas a lower bound for $\RC(5,5,5)$ comes from partitioning the vertices of $K_{6}$ into $A_1,A_2,A_3$ with $|A_i|=2$ and coloring edges in $A_i$ and in $[A_i,A_{i+1}]$ with color $i$.

Another example is $\PM(4,3,3,3)=5=\RC(4,3,3,3)$. Coloring the edges of a $K_4$ by coloring a triangle with colors 2,3,4 and coloring the other edges with color 1 provides a lower bound for $\PM(4,3,3,3)$; whereas coloring a triangle with color 1 and coloring the other edges with three different colors provides the lower bound for $\RC(4,3,3,3)>4$.



On the other hand we have many examples where $\PM(p_1, \dots, p_r)>\RC(p_1, \dots, p_r)$, such as $\PM(4,4,4)=6>5=\RC(4,4,4)$. So this raises the following question.

\begin{question}
Under what circumstances do we have $\PM(p_1, \dots, p_r)=\RC(p_1, \dots, p_r)$?
\end{question}

\medskip
\noindent{{\bf Acknowledgement.} Thanks to Andr\'e E. Kezdy for his help to obtain the thesis \cite{SCO} and to J\'acint Szab\'o for helping with the reference \cite{LV}.  We also thank the referee for some helpful remarks.

\end{document}